\documentclass{amsart}
\usepackage{amsmath,amsthm,amsfonts,amssymb,latexsym,mathrsfs,graphicx}

\usepackage{hyperref}

\usepackage{enumerate}
\usepackage[shortlabels]{enumitem}
\usepackage{color}

\newtheorem{theorem}{Theorem}[section]

\theoremstyle{definition}

\newcommand{\Irr}{{\mathrm {Irr}}}

\newcommand{\cd}{{\mathrm {cd}}}

\newcommand{\acd}{{\mathrm {acd}}}

\newcommand{\PSL}{{\mathrm {PSL}}}

\newcommand{\SL}{{\mathrm {SL}}}

\begin{document}
	
\title{On average character degree of some irreducible characters  of a finite group}

\author[Z. Akhlaghi et al.]{Zeinab Akhlaghi}
\address{Zeinab Akhlaghi, Faculty of Math. and Computer Sci., \newline Amirkabir University of Technology (Tehran Polytechnic), 15914 Tehran, Iran.\newline
	School of Mathematics,
	Institute for Research in Fundamental Science(IPM)
	P.O. Box:19395-5746, Tehran, Iran.}
\email{z\_akhlaghi@aut.ac.ir}

\thanks{
	The  author  is supported by a grant from IPM (No. 1400200028).}

\begin{abstract}Let $G$ be a finite group and $N$ be a non-trivial normal subgroup of $G$, such that the average character degree of irreducible characters in $\Irr(G|N)$ is less than  or equal to $16/5$. Then we prove that $N$ is solvable. Also,  we prove the solvability of  $G$,  by assuming   that the average character degree of irreducible characters in $\Irr(G|N)$ is  strictly less than $16/5$. We show that the bounds are sharp. 

\end{abstract}


\maketitle

\section{Introduction}
The set of character degrees of a finite group  has been an object of considerable interest for a long time. In some of researches, the authors associate 
  some invariants   to  the character degrees  of a finite group and then they study the influence of this invariant  on the structure of the  group. As an example, the average character degree of finite groups has been studied  by a lot of researchers. We denote by $\acd(G)$ the  average character degree of  a finite group $G$, which is defined as following:

$$\acd(G)=\frac{\sum\limits_{\chi\in \Irr(G)}\chi(1)}{|\Irr(G)|}.$$

The influence of $\acd(G)$ on the structure of $G$ has attracted the attention of lots of authors. In \cite{isaacs}, it is conjectured that if $\acd(G)<\acd(A_5)=16/5$, then $G$ is solvable.  The conjecture proved by A. Moret\'o and H.N. Nguyen, in \cite{moreto}. Let $N$ be a non-trivial normal subgroup of $G$, $\Irr(G|N)$ be  the set of those irreducible characters of $G$  whose kernels do not contain $N$ and $\cd(G|N)$ be the set of character degrees of irreducible characters in $\Irr(G|N)$.  By $\acd(G|N)$ we mean the average character degree of irreducible characters in $\Irr(G|N)$.  We prove that if $\acd(G|N)\leq 16/5$, then $N$ is solvable, moreover if $\acd(G|N)<16/5$, then $G$ is solvable. Note that the bound is sharp, in fact if we take $G_n\cong C_n\times A_5$, for some integer $n$, then $\acd(G_n|G_n)=(16n-1)/(5n-1)$. Therefore,  
the ratio $\acd(G_n|G_n)$ would converge to $16/5$ as $n$ tends to infinity. Also $\acd(A_5\times C_2|C_2)=16/5$, which means that, to prove the   solvability of $G$ we need $\acd(G|N)$ to be strictly less than $16/5$.

To obtain  the main result, we  follow  the techniques in \cite{moreto}.  It is worth mentioning that  $\acd(G|N)< 16/5$ does not infer $\acd(N)<16/5$ or vice versa. For instance, if we take $G$ to be a Frobenius group $2^6:63$ and $N$ to be a Frobenius group $2^6: 7$. Then $\acd(G|N)=117/55<16/5$ and $\acd(N)= (9\times 7+7)/16> 16/5$.
For the inverse we just need to take a look at the group $G=\SL_2(5)$ and $N={\bf Z}(G)$. Then $\acd(N)=1$, while $\acd(G|N)=(2+2+4+6)/4>16/5$. First example says that to prove the solvability of $N$, when $\acd(G|N)<16/5$, we can not refer to the  Theorem A of \cite{moreto}, directly. 

Throughout  the paper $G$ is a finite group.  Let $N\unlhd G$  and $\lambda\in \Irr(N)$. By $\acd(G|\lambda)$ we mean the average character degree of  irreducible characters of $G$ above $\lambda$.   For the rest of notation, we follow \cite{Isaacs}.

\section{Main Results }


\begin{theorem}
	Let $G$ be a finite group and  $N$ be a  non-trivial normal subgroup of $G$. If $\acd(G|N)\leq  16/5$, then $N$ is solvable.
\end{theorem} 
\begin{proof}
	Assume, on the contrary,   $G$ is an example  with minimal order, such that $G$ has  a normal  non-solvable  subgroup $N$ and $\acd(G|N)\leq 16/5$.  Thus,
	$$\acd(G|N)=\frac{\sum \limits_{\chi\in \Irr(G|N)}\chi(1)}{|\Irr(G|N)|}\leq  16/5.$$ 
	
	Let $n_{d}(G|N)$ be the number of characters in $\Irr(G|N)$ of degree $d\geq 1$. Then, $\sum \limits_{\chi\in \Irr(G|N)}\chi(1)=\sum\limits_{d\geq 1}dn_d(G|N)$ and $|\Irr(G|N)|=\sum\limits_{d\geq 1}n_d(G|N)$.  So by the above inequality we have
	
	$$\sum\limits_{d\geq 4}(5d-16)n_{d}(G|N)\leq  11n_1(G|N)+ 6n_2(G|N) +n_3(G|N). \   \   \   \  \  \  \  \  (*)$$
	
	First, we claim that there is no non-solvable minimal normal subgroup of $G$ contained in $N'$. On the contrary, let $M\leq N'$ be a  non-solvable minimal normal subgroup of $G$. Then $M$ is a direct product of $k$ copies of a  non-abelian finite simple group $S$, for some integer $k$. By the hypothesis, we have $M\leq G'$ and so $M$ is contained in the kernel of every linear character of $G$. We show that $M$ is contained in the kernel of  every irreducible  character of $G$  of degree $2$. Let $\chi\in \Irr(G)$ such that $\chi(1)=2$. Since, non-abelian finite simple groups do not have any irreducible character of degree $2$ and the only linear character of a simple group is the principle character, then  $\chi_M=2.1_M$. Therefore $M$ lies in the kernel of $\chi$, as wanted. Hence $n_d(G|N)=n_d(G/M|N/M)$, for  $d=1,2$. We break the proof of the claim in two following cases:  
	
	\bigskip

	{\bf Case 1.} Let $S\ncong\PSL_2(q)$, for $q=5,7$. It is well-known that  $S$  and so $M$ do not have any irreducible character of degree $3$. So, by the similar discussion as above we have $n_3(G|N)=n_3(G/M|N/M)$ and we conclude that 
	$$n_1(G|N)+n_2(G|N)+n_3(G|N)=$$$$n_1(G/M|N/M)+n_2(G/M|N/M)+n_3(G/M|N/M)\leq |\Irr(G/M|N/M)|.$$
	On the other hand, by \cite[Lemma 1]{moreto},  $M$ has an irreducible character $\theta$ with degree  $d_0\geq 6$ which is extendable to $G$. Then, by Gallagher's theorem (see \cite[Corollary 6.17]{Isaacs}), we have   $|\Irr(G/M)|\leq \sum\limits_{d_0\mid d}n_{d}(G|M)\leq  \sum\limits_{d\geq 6}n_d(G|M)$, which implies that
	
	$$|\Irr(G/M|N/M)|\leq|\Irr(G/M)|\leq \sum\limits_{d\geq 6}n_d(G|M) \leq \sum\limits_{d\geq 6}n_d(G|N). $$
	
	Therefore,  $$n_1(G|N)+n_2(G|N)+n_3(G|N)\leq \sum\limits_{d\geq 6}n_d(G|N).$$ Hence,  
	$$\sum\limits_{d\geq 4}(5d-16)n_{d}(G|N)\geq \sum\limits_{d_0 \mid d}(5d-16)n_{d}(G|N) \geq (5d_0-16)(n_1(G|N)+ n_2(G|N) +n_3(G|N)) \geq $$$$ 11n_1(G|N)+6n_2(G|N)+n_3(G|N).$$
	
	The last inequality,  turns to  equality if and only if $n_i(G|N)=0$, for  each $i\in \{1,2,3\}$,  which is  contradicting $(*)$, as in this case $\sum\limits_{d\geq 6}(5d-16)n_d(G|N)>0$. So in any case we get a contradiction by $(*)$.    
	
	\bigskip
	
	{\bf Case 2.} Let $S\cong \PSL_2(q)$, for $q=5,7$. Then, $S$  has exactly two irreducible characters of degree $3$ and $M$ has two irreducible characters of degrees $d_1$ and $d_2$, where  $(d_1,d_2)=(5^k,4^k)$ or $(8^k,7^k)$, when $q=5$ or $7$,  respectively,  which are extendable to $G$. Recall that by  the discussion proceeding Case 1, $n_d(G|N)=n_d(G/M|N/M)$, when $d=1,2$. On the other hand, using  Gallagher's theorem we have $n_{t}(G/M|N/M)\leq n_{t}(G/M)= n_{t.d_i}(G|M)\leq n_{t.d_i}(G|N)$, for $i=1,2$ and integer $t$. Hence,   $$n_1(G|N)+n_2(G|N)+n_3(G/M|N/M)\leq n_{d_1}(G|N) +n_{2.d_1}(G|N)+n_{3.d_1}(G|N).$$ Noting that $5d_1- 16\geq 9$, we come to the following conclusion,

	$$9n_1(G|N)+6n_2(G|N)+n_3(G/M|N/M)\leq 9(n_{d_1}(G|N) +n_{2.d_1}(G|N)+n_{3.d_1}(G|N))$$$$\leq  
	\sum\limits_{d\geq d_1}(5d-16)n_d(G|N).$$
	
	We show that  $n_3(G|M)\leq 2kn_1(G)$. Noticing that $M$ does not have any character of degree $2$ and any non-principal linear character,   we deduce that
	   $\chi_M$ is an irreducible character of $M$, for every character $\chi$ of  degree $3$ in $\Irr(G|M)$.    On the other hand,  $M$ has exactly $2k$ characters of degree $3$ and so,  by Gallagher's theorem, there exist at most $2kn_1(G)$ characters of degree $3$ in $\Irr(G|M)$, as wanted.
	
	Recalling that $n_1(G)=n_1(G/M)=n_{d_2}(G|M)\leq n_{d_2}(G|N)$, and together with the above conclusion we infer that 
	
	$$\sum\limits_
	{d\geq d_2}
	(5d - 16)n_d(G|N) \geq (5.d_2 - 16)n_{d_2}(G|N) +\sum\limits_{d\geq d_1}
	(5d - 16)n_d(G|N)
	\geq  $$$$
	(5d_2 - 16)n_1(G) + 9n_1(G|N) + 6n_2(G|N) +  n_3(G|N) - 2kn_1(G)
	\geq $$$$ 2n_1(G) + 9n_1(G|N) + 6n_2(G|N) + n_3(G|N) > 11n_1(G|N)+6n_2(G|N)+n_3(G|N),$$
	
	as $5d_2-16\geq 2+2k$, for $k\geq 1$, which is  contradicting $(*)$.

	Therefore, our claim is proved, hence we  may assume every normal minimal subgroup of $G$, contained in $N'$ is solvable.  Let $M\unlhd G$ contained in $N$  be minimal such that $M$ is non-solvable. Notice that  $M$ is a perfect group contained in the last term of derived series of $N$. Let $T\leq M$, such that $T$ is a minimal normal subgroup of $G$.  In addition, if $[M, R]\not =1$, we assume $T\leq [M,R]$, where $R$ is the solvable radical subgroup of $M$, and if it is possible we assume $T$ has order $2$. Therefore $T\leq M'\leq G'$ and $N/T$ is non-solvable. As, $G$ is a counterexample of minimal order, we  have $\acd(G/T|N/T)> 16/5$ which means
	
	$$\sum\limits_{
		d\geq 4}(5d-16)n_d(G/T|N/T) > 11n_1(G/T|N/T) + 6n_2(G/T|N/T) + n_3(G/T|N/T).\  \  \  \  (**)$$  
	Since $T$ lies  in the kernel of every linear character of $G$, then $n_1(G/T|N/T)=n_1(G|N)$ and considering $(*)$ and $(**)$  we have
	$$6n_2(G/T|N/T) + n_3(G/T|N/T)<6n_2(G|N)+n_3(G|N).$$
	
	Hence $\Irr(G|T)$ contains a character of degree $2$ or $3$, say $\chi$. Replacing $N$ by $T$ and repeating the same discussion  in the last paragraph of  page 458 and the first three paragraphs of page 459 of \cite{moreto}, we see that  if $K=\ker(\chi)$, then  $G/K$ is a primitive group of degree $2$ or $3$ and also  $G/C\cong A_5$, $\PSL_2(7)$ or $A_6$, where $C/K=Z(G/K)$. In addition,  $G=MC$ is a central product with central  subgroup ${\bf Z}(M)=M\cap C$. Moreover, the cases $G/C\cong A_6$ and $\PSL_2(7)$ only happen when $\chi(1)=3$.

	Note that  $M/(M\cap C)\cong G/C$, so by the choice of $T$, we get that $T\leq M\cap C={\bf Z}(M)$. As, ${\bf Z}(M)$ lies in the Schur multiplier of $M/{\bf Z}(M)$, we deduce that $T\cong C_2$ when $M/{\bf Z}(M)\cong A_5$ or $\PSL_2(7)$,  and  $T\cong C_2$ or $C_3$, when  $M/{\bf Z}(M)\cong A_6$. 
	
	First let  $M/{\bf Z}(M)\cong A_6$ or $\PSL_2(7)$ and $T\cong C_2$. Then,  $\chi_M$ is an irreducible character of $M$ of degree $3$. In both cases, the kernel of  irreducible  characters of $M$ of degree $3$  contains $T\cong C_2$, which is impossible. Hence, $G/C\cong A_5$ or $A_6$ and when the second case occurs, $T\cong C_3$ and so ${\bf Z}(M)\cong C_3$, since we assumed $T$ has order $2$, if it is possible.
	
	Let $G/C\cong A_5$. Then $T\cong {\bf Z}(M)\cong C_2$,  $M\cong \SL_2(5)$ and  $\Irr(G|T)= \Irr(G|\lambda)$, where $\lambda$ is the only  non-trivial character of $T$. By \cite[Lemma 2]{moreto}, we have
	$\acd(G|T)=\acd(G|\lambda)	\geq \acd(M|\lambda)=14/4 > 16/5$.  Recall that, as $G$ is a   counterexample with minimal order, $\acd(G/T|N/T) > 16/5$, hence by\cite[Lemma 2.1]{qian}, $\acd(G|N) > 16/5$, which is a contradiction.

	Let $G/C\cong A_6$ and $T\cong {\bf Z}(M)\cong C_3$. 
	Similar to the above case, $\acd(G|T)\geq \acd(M|\lambda)=36/5 > 16/5$.  Also by the minimality of $G$ we have $\acd(G/T|N/T) > 16/5$. Therefore, $\acd(G|N) > 16/5$, a contradiction. So, the proof is complete.
		\begin{small} $\blacksquare$
	\end{small}
\end{proof}
\begin{theorem}
	Let $G$ be a finite group and $N$ be a  non-trivial normal subgroup of $G$ such that $\acd(G|N)< 16/5$. then $G$ is solvable. 
\end{theorem}

\begin{proof}
	Assume $G$ is not solvable. By the previous theorem $N$ is solvable. Hence, $G/N$ is non-solvable. Therefore, for some prime $r$, $G/N$ is not $r$-solvable. Then by \cite[Theorem 3.2]{qian},   we have $\acd(G|\lambda)\geq \lambda(1)f(r)$, and so $\acd(G|\lambda)\geq f(r)$, for each $\lambda\in \Irr(N)$, where $f(r)$ is a function defined in \cite{qian}. As, $f(r)\geq 16/5$, for each  prime $r\geq 2$ and considering \cite[Lemma 2.1]{qian},  we deduce that $\acd(G|N)\geq 16/5$,  a contradiction. Now, the proof is complete.
	\begin{small} $\blacksquare$
		\end{small}
\end{proof}

\end{document}